\documentclass[12pt]{article}%
\usepackage{amsmath}
\usepackage{amsfonts}
\usepackage{amssymb}
\usepackage{graphicx}%
\providecommand{\U}[1]{\protect\rule{.1in}{.1in}}
\newtheorem {theorem}{Theorem}

\newtheorem {conclusion}[theorem]{Conclusion}

\newtheorem {proposition}[theorem]{Proposition}
\newtheorem {remark}[theorem]{Remark}

\newenvironment {proof}[1][Proof]{\noindent \textbf {#1.} }{\ \rule {0.5em}{0.5em}}

\newtheorem {rmk}{Remark}

\newtheorem {prf}{Proof}

\begin{document}

\setlength{\textwidth}{6.6in}
\setlength{\topmargin}{-0.9in}
\setlength{\textheight}{1.20\textheight}
\setlength{\oddsidemargin}{0.25in}
\setlength{\evensidemargin}{-0.25in}

\newcommand{\NOT}[1]{}
\newcommand{\pa}{\par\medskip}
\newcommand{\eps}{\varepsilon}

\title{\textbf{Accumulated Random Distances in High Dimensions -- Ways of Calculation}}
\author{Eliahu Levy\\Department of Mathematics\\Technion -- Israel Institute of Technology\\Technion City, Haifa 3200003, Israel\\(eliahu@math.technion.ac.il)}
\date{March 24, 2020}

\maketitle

\begin{abstract}
In this note we refine and improve some of the calculations in our 2019 article with Yair Censor (Applied Mathematics and Optimization, accepted for publication) where an analysis of the superiorization method is made
via the principle of concentration of measure. Some paragraphs there are repeated here for the sake of completeness. Yet, for the case of accumulating `steps' on the sphere, reference to distances as done there is replaced by reference to the angles, which makes simpler expressions. The treatment here of the action of a random transformation is also rather `cleaner'. For some standard deviations, precise inequalities are here obtained rather then just
$O$-expressions. Some further settings are mentioned, of no direct interest as per the latter article, showing results that similar calculations yield.
\end{abstract}

\section{Concentration of Measure in High Dimensional Euclidean
Spaces}
Concentration of (Probability) Measure is the phenomenon that probability is highly concentrated near one value,
thus near the expectation.
\pa
A classical example is \emph{The Law of Large Numbers}, stating the concentration of measure of the sum (or average) of many i.i.d.\ stochastic variables.
\pa
We will focus on peculiar concentration of measure phenomena in $N$-dimensional Euclidean
spaces when $N$ is large.
\pa
To see, as an example, why Euclidean spaces $\mathbb{R}^N$ of high dimension $N$ should lead to concentration of measure phenomena, consider the unit sphere in $\mathbb{R}^N$.
\pa
`Partitioning' it into layers in parallel hyperplanes orthogonal to some fixed unit vector, the layer distanced $t$ from the central hyperplane through $0$ is a sphere of one less dimension with radius $\sqrt{1-t^2}$. Its $N-2$-`area' will be proportional to the $N-2$ power of its radius, hence to $(1-t^2)^{(N-2)/2}$.
\pa
And for a vector uniformly distributed on the sphere, the probability for distance $t$ from the central hyperplane will have density proportional to that `$N-2$-area' multiplied by $(1-t^2)^{-1/2}$ (-- the latter due to the need to project $dt$ onto the sphere), thus to $(1-t^2)^{(N-3)/2}$.
\pa
$N$ being `big', if $1-t^2$ is even slightly less than $1$, the `big' exponent $N-3$ will make
$(1-t^2)^{(N-3)/2}$ small. Indeed, $(1-t^2)^{(N-3)/2}$, when not negligible, is approximately $e^{-Nt^2/2}$, significant only when $t=O(1/\sqrt{N})$.
\pa
So we have the somewhat strange-sounding fact (later we'll also derive that differently):
\pa
\textbf{For two vectors in $\mathbb{R}^N$ with independent uniformly distributed directions, it is highly unlikely that the angle between them is far from $90^\circ$. In fact, the deviation from $90^\circ$ is with high probability $O(1/\sqrt{N})$}.

\section{Using the Normal Distribution} \label{sc:ConcMeas}
Yet a main vehicle in studying $\mathbb{R}^N$ is the \emph{Normal Distribution}. (Indeed, as we saw above, our $(1-t^2)^{(N-3)/2}$  above `ended' as approximately `normal', as the Central Limit Theorem would make us expect.)
\pa
The standard normal distribution on $\mathbb{R}$,\quad $\mathcal{N}(0,1)$, has distribution
$$\dfrac1{\sqrt{2\pi}}e^{-\frac12 x^2}\,dx.$$
\pa
Thus in the Euclidean $\mathbb{R}^N$, i.e.\ endowed with $\ell^2$ norm, letting the $N$ coordinates to be independent identically distributed $\sim\mathcal{N}(0,1)$ will give the distribution on $\mathbb{R}^N$,
$$(2\pi)^{-N/2}e^{-\frac12\|x\|^2}\,dx_1\ldots\,dx_N.$$
\pa
hence \emph{invariant under all linear orthogonal transformations of $\mathbb{R}^N$}.
\pa
In particular any projections of it on $N$ orthogonal unit vectors are also $N$ independent $\sim\mathcal{N}(0,1)$.
\pa
To see what one may obtain thus, for our above $x=(x_1,\ldots,x_n)$ with coordinates i.i.d.\ $\sim\mathcal{N}(0,1)$, we have, by the orthogonal invariance,
\pa
\textbf{For any fixed unit vector $u\in\mathbb{R}^N$, $u\cdot x$ is $\sim\mathcal{N}(0,1)$}
\pa
On the other hand, $\|x\|^2=x_1^2+\ldots+x_N^2$. These squares are still i.i.d, therefore, as always
with sums of i.i.d., $\|x\|^2$ has variance $N$ times the variance of the square of a $\sim\mathcal{N}(0,1)$, hence has standard deviation $\sqrt{N}$ times the standard deviation of the square a $\sim\mathcal{N}(0,1)$ -- some number independent of $N$.
\pa
And, of course, the expectation of $\|x\|^2$ is $N$ (because for a single coordinate it is the variance of a $\sim\mathcal{N}(0,1)$ itself = $1$).
\pa
So \emph{$\|x\|^2$ is highly concentrated near its expectation $N$} -- with relative error $O(1/\sqrt{N})$, hence \emph{$\|x\|$ is highly concentrated near $\sqrt{N}$}.
\pa
Thus, speaking somewhat loosely, for the above $x$ with i.i.d.\ $\sim\mathcal{N}(0,1)$ coordinates, \emph{$\left(1/\sqrt{N}\right)x$ will rarely deviate from $1$.}
\pa
And of course, by the orthogonal invariance, $x/\|x\|$ must is distributed uniformly on the unit sphere.
\emph{So we have here a way to get this uniform distribution}.
\pa
But $x/\|x\|$ is rarely much different from $\left(1/\sqrt{N}\right)x$, thus the latter (Gaussian!) distribution may also serve as an approximation to the uniform distribution on the unit sphere -- the unit point `wanders' along the radius, but `rarely more than $O(1/\sqrt{N})$'.
\pa
And we can deduce the above fact about the angle between vectors:
\pa
The Gaussian $\left(1/\sqrt{N}\right)x$ differs from $x/\|x\|$, which gives the uniform distribution on the unit sphere, just by a relative $O(1/\sqrt{N})$ `perturbation' of the vector.
\pa
The cosine of its angle with a fixed unit vector $u$ is $\left(1/\sqrt{N}\right)(u\cdot x)$ and $(u\cdot x)$ is $\sim\mathcal{N}(0,1)$, making that cosine (= the sine of the difference with $90^\circ$) approximately $\sim\left(1/\sqrt{N}\right)\mathcal{N}(0,1)$.

\section{The norm of the sum of vectors with given norms\label{sc:Sum}}
As a further example of the use of the normal distribution, suppose we are given $M$ vectors $y_1,y_2,\ldots,y_M$ of known norms $d_1,d_2,\ldots.d_M$ in a high-dimensional $E^{N}$. What should we expect the norm of their sum to be?
\pa
That can be answered: take the direction of each of them distributed uniformly on $S^{N-1}$, \emph{even conditioned on fixed valued for the others}. In other words, take them independent, each with direction distributed uniformly. This
can be constructed by taking random $M$ vectors in $E^{N}$ (that is, a random $M\times N$ matrix), with entries i.i.d.\ $\sim{\mathcal{N}}$, dividing them by $\sqrt{N}$, then by their norm (now highly concentrated near $1$)
and then multiplying them (i.e.\ the the columns of the matrix), by
by $d_1,d_2,\ldots,d_M,$ respectively.\pa

Then the sum $\sum_{i=1}^{M}y_i$, if we ignore the division by the norm, is $1/\sqrt{N}$ times the random matrix applied to the vector $(d_1,d_2,\ldots,d_{M})$.\pa

But the distribution of the random matrix is invariant with respect to any transformation which is orthogonal with respect to the Hilbert-Schmidt norm -- the square root of the sum of squares of the entries
(i.e.,\ $\Vert T\Vert_{HS}:=\sqrt{\operatorname*{tr}(T^{\prime}\cdot T)}=
\sqrt{\operatorname*{tr}(T\cdot T^{\prime})}$,\, \, $T^{\prime}$ denoting the transpose and $\operatorname*{tr}$ standing for the trace, see, e.g., \cite{bell16}).\pa

In particular, the distribution of the sum is the same as that of $1/\sqrt{N}$ times $\sqrt{d_1^2+d_2^2+\cdots+d_M^2}$ times the random matrix applied to $(1,0,\ldots,0)$, which is, of course, distributed with independent $\sim{\mathcal{N}}$ entries, thus, with norm concentrated near $\sqrt{N}$. (With relative deviation $O(1/\sqrt{N})$.)
This leads to the following conclusion.\pa

\begin{conclusion}\label{conc:SpaceVectors}
\label{conc:Sum} For $M$ vectors $y_1,y_2,\ldots,y_M$ of known norms $d_1,d_2,\ldots.d_M$, in $E^{N}$, (and taking their directions as random, distributed independently uniformly), we have that $\Vert\sum_{i=1}^{M}y_{i}\Vert$ is near $\sqrt{d_1^2+d_2^2+\cdots+d_M^2}$ with almost full probability (With relative deviation $O(1/\sqrt{N})$.)
\end{conclusion}

We shall return to that in \S{\ref{sc:MarkovFlat}}.

\section{The accumulation of given distances on the unit sphere\label{sc:Sphere}}
Now ask a similar question \emph{on the unit sphere}: what should we expect the angle of a chord made of $M$ chords, \emph{the differences between consecutive elements in a sequence of points on the unit sphere $S^{N-1}\subset E^{N}$}, of given angles $\theta_1,\theta_2,\ldots.\theta_M$.
\pa
Note that here the angles are the natural `distances'. Indeed they naturally measure the distance on geodesics
 -- great circles.
\pa
Denote by $\omega_{N-1}$ the normalized to be probability (i.e.,\ of total mass $1$) uniform measure on $S^{N-1}$.

\begin{remark}
By symmetry, for $x=(x_1,x_2,\ldots,x_N)\in S^{N-1}$, $\int x_{k}%
^2\,d\omega_{N-1}$ is the same for all $k$. Of course, their sum is
$\int1\,d\omega_{N-1}=1$. Therefore,
\begin{equation}
\int x_{k}^2\,d\omega_{N-1}=\frac1{N},\qquad k=1,2,\ldots,N.
\label{eq:integral}%
\end{equation}
Hence, for a polynomial of degree $\leq2$ on $E^{n}$:
\begin{equation}
p(x)=\langle Qx,x\rangle+2\langle a,x\rangle+\gamma, \label{eq:poly}%
\end{equation}
where $Q$ is a symmetric $N\times N$ matrix, $a\in E^{N}$ and $\gamma\in E$,
we will have
\begin{equation}
\int p(x)\,d\omega_{N-1}=\dfrac1{N}\operatorname*{tr}\,Q+\gamma.
\end{equation}

\end{remark}

Now, for some fixed $0\leq \theta\leq\pi$, the set of points in $S^{N-1}$ of
angular distance $\theta$ from some fixed vector $u\in S^{N-1}$ is the $(N-2)$-sphere
$\subset S^{N-1},$ $\Sigma(u,\theta)$ given by
\begin{equation}
\Sigma(u,\theta):=\cos\theta\cdot u+\sin\theta\cdot{S^{N-2}}_{u^{\bot}},\label{eq:sphere}%
\end{equation}
where ${S^{N-2}}_{u^{\bot}}$ stands for the unit sphere in the hyperplane
perpendicular to $u.$
\pa
In our scenario, one performs a \emph{Markov chain}$,$
see, e.g., \cite{markov2000}\emph{.} Starting from a point $u_0$ on
$S^{N-1}$, move to a point $u_1\in\Sigma(u_0,\theta_1)$ uniformly
distributed there. Then, from that $u_1$, to a point $u_2\in\Sigma
(u_1,\theta_2)$ uniformly distributed there, and so on, until one ends with
$u_M$. We would like to find the expected cosine of the angle between $u_0$ and
$u_M$, namely
$\mathbb{E}\left[\langle u_M,u_0\rangle\right]$.
\pa
If we denote by ${\mathcal{L}}_\theta$ the operator mapping a function $p$ on
$S^{N-1}$ to the function whose value at a vector $u\in S^{N-1}$ is the
average of $p$ on $\Sigma(u,\theta)$, then ${\mathcal{L}}_{\theta_k}(p)$ evaluated at
$u$ is the expectation of $p$ at the point to which $u$ moved in the $k$-th
step above. Hence, in the above Markov chain, the expectation of $p(u_M)$
is
\begin{equation}
\left({\mathcal{L}}_{\theta_M}{\mathcal{L}}_{\theta_{M-1}}\cdots{\mathcal{L}}_{\theta_1}(p)\right)(u_0).
\end{equation}
Thus, what we are interested in is
\begin{equation}
\mathbb{E}\left[\langle u_M,u_{0}\rangle\right]  =
{\mathcal{L}}_{\theta_M}\mathcal{L}_{\theta_{M-1}}\cdots\mathcal{L}_{\theta_1}(\langle x,u_0\rangle)\Big\vert_{x=u_0}.
\end{equation}
\pa
So, let us calculate ${\mathcal{L}}_{\theta}(p)$ for polynomials of degree $\leq2$
as in (0\ref{eq:poly}). In performing the calculation, assume
$u=(1,0,\ldots,0)$. For $x=(x_1,x_2,\ldots,x_N)\in E^{N}$ write
$y=(x_2,x_3,\ldots,x_N)\in E^{N-1}$. In (\ref{eq:poly}) write $a=(a_1,b)$ where
$b=(a_2,a_3,\ldots a_N)\in E^{N-1}$ and
\begin{equation}
Q=\left(
\begin{array}
[c]{cc}%
\eta & c^{\prime}\\
c & Q_1%
\end{array}
\right)  ,
\end{equation}
where $Q_1$ is a symmetric $(N-1)\times(N-1)$ matrix, $c\in E^{N-1}$
and $\eta\in E$. Note that for our $u=(1,0,\ldots,0)$,\,
\, \,  $a_1=\langle a,u\rangle$,\, \,
\, \ $\eta=\langle Qu,u\rangle$ and $\operatorname*{tr}\,Q_1=
\operatorname*{tr}\,Q-\eta=\operatorname*{tr}\,Q-\langle Qu,u\rangle$.
\pa
Then, for $p$ as in in (\ref{eq:poly}),
\begin{equation}
p(x)=\eta x_1^2+2x_1\langle c,y\rangle+\langle Q_1y,y\rangle
+2a_1x_1+2\langle b,y\rangle+\gamma.
\end{equation}
Hence, taking account of (\ref{eq:sphere}) for $u=(1,0,\ldots,0)$, and using
(\ref{eq:integral}),
\begin{eqnarray*}
&& \left({\mathcal{L}}_\theta p\right)(u)=\left({\mathcal{L}}_\theta p\right)(1,0,\ldots,0)=\\
&& =\cos^2\theta\cdot\eta+\frac1{N-1}\sin^2\theta\,\operatorname*{tr}\,Q_1+
2\cos\theta\,a_1+\gamma=\\
&& =\cos^2\theta\,\langle Qu,u\rangle+
\frac1{N-1}\sin^2\theta\,\left(\operatorname*{tr}Q-\langle Qu,u\rangle\right)
 +2\cos\theta\,\langle a,u\rangle+\gamma=\\
&& =\left(\cos^2\theta-\frac1{N-1}\sin^2\theta\right)\langle Qu,u\rangle+
 \frac1{N-1}\sin^2\theta\operatorname*{tr}Q+
 2\cos\theta\,\langle a,u\rangle+\gamma.
\end{eqnarray*}
which, by symmetry, will hold for any $u\in S^{N-1}$. In particular, we find,
as should be expected, that
\begin{eqnarray*}
&& \int(\mathcal{L}_\theta(p))(x)\,d\omega_{N-1}\\
&& =\frac1{N}\left(\cos^2\theta-\frac1{N-1}\sin^2\theta\right)\,\operatorname*{tr}\,Q+
\frac1{N-1}\sin^2\theta\operatorname*{tr}\,Q+\gamma\\
&&  =\frac1{N}\operatorname*{tr}\,Q+\gamma=\int p(x)\,d\omega_{N-1}.
\end{eqnarray*}
We are interested, for some fixed $u\in S^{N-1}$, in
\begin{equation}
p(x)=\langle u,x\rangle.
\end{equation}
Then there is no $Q$ term, so one has
\begin{equation}
\Big({\mathcal{L}}_\theta\big(\langle u,x\rangle\big)\Big)(u)=
\cos\theta\cdot\langle u,x\rangle.
\end{equation}
Consequently,
\begin{eqnarray*}
&&  \mathbb{E}\left[\langle u_M,u_0\rangle\right] =
\left({\mathcal{L}}_{\theta_{M}}{\mathcal{L}}_{\theta_{M-1}}\cdots{\mathcal{L}}_{\theta_1}
(\langle x,u_0\rangle)\right)\vert_{x=u_0}=\\
&& =\prod_{i=1}^M\cos\theta_i\cdot(\langle u_0,x\rangle)\Big\vert_{x=u_0}=
\prod_{i=1}^{M}\cos\theta_i.
\end{eqnarray*}
\pa
We also assess the standard deviation, which is
\begin{equation}
\sigma=\sqrt{\left({\mathcal{L}}_{\theta_{M}}{\mathcal{L}}_{\theta_{M-1}}\cdots
{\mathcal{L}}_{\theta_1}\right)(\langle u_0,x\rangle^2)\Big\vert_{x=u_0}
-\left(\prod_{i=1}^M\cos\theta_i\,\right)^2}.
\end{equation}
Here $p(x)=\langle a,x\rangle^2$, so there is only the $Q$ term with
$Q(x):=\langle a,x\rangle^2$. Then $\operatorname*{tr}\,Q=\Vert a\Vert^2$,
and we find
\begin{eqnarray*}
&& \Big({\mathcal{L}}_{\theta}\big(\langle a,x\rangle^2\big)\Big)(u)=\\
&& =\left(\cos^2\theta-\frac1{N-1}\sin^2\theta\right)\langle a,u\rangle^2+
\frac1{N-1}\sin^2\theta\,\Vert a\Vert^2.
\end{eqnarray*}
Consequently, for $a=u_0$ (note $\Vert u_0\Vert^2=1$),
\begin{eqnarray*}
&& \sigma^2=\left({\mathcal{L}}_{\theta_M}{\mathcal{L}}_{\theta_{M-1}}\cdots
{\mathcal{L}}_{\theta_1}\right)(\langle u_0,x\rangle^2)\Big\vert_{x=u_0}-
\left(\prod_{i=1}^M\cos\theta_i\right)^2=\\
&& =-\left(\prod_{i=1}^M\cos\theta_i\right)^2+
\prod_{i=1}^M\left(\cos^2\theta_i-\frac1{N-1}\sin^2\theta_i\right)\\
&& +\frac1{N-1}\left[\sin^2\theta_1
+\sin^2\theta_2\left(\cos^2\theta_1-\frac1{N-1}\sin^2\theta_1\right)\right.\\
&& +\sin^2\theta_3\left(\cos^2\theta_2+\frac1{N-1}\sin^2\theta_2\right)
\left(\cos^2\theta_1-\frac1{N-1}\sin^2\theta_1\right)\\
&& \left.+\cdots+\sin^2\theta_M\prod_{i=1}^{M-1}\left(\cos^2\theta_i-\frac1{N-1}\sin^2\theta_i\right)\right],
\end{eqnarray*}
But,
\begin{eqnarray*}
&& \prod_{i=1}^M\left(\cos^2\theta_i-\frac1{N-1}\sin^2\theta_i\right)\\
&& = -\frac1{N-1}\sin^2\theta_M\cdot\prod_{i=1}^{M-1}\left(\cos^2\theta_i-\frac1{N-1}\sin^2\theta_i\right)\\
&& -\cos^2\theta_M\cdot\frac1{N-1}\sin^2\theta_{M-1}
\cdot\prod_{i=1}^{M-2}\left(\cos^2\theta_i-\frac1{N-1}\sin^2\theta_i\right)\\
&& -\cos^2\theta_M\cos^2\theta_{M-1}\cdot\frac1{N-1}\sin^2\theta_{M-2}
\cdot\prod_{i=1}^{M-3}\left(\cos^2\theta_i-\frac1{N-1}\sin^2\theta_i\right)\\
&& - \cdots - \prod_{i=3}^M\cos^2\theta_i\cdot\frac1{N-1}\sin^2\theta_2
\left(\cos^2\theta_1-\frac1{N-1}\sin^2\theta_1\right)\\
&& - \prod_{i=2}^M\cos^2\theta_i\cdot\frac1{N-1}\sin^2\theta_1 + \prod_{i=1}^M\cos^2\theta_i.
\end{eqnarray*}
So we find,
\begin{eqnarray*}
&& \sigma^2=\frac1{N-1}\left[\sin^2\theta_{M-1}\left(1-\cos^2\theta_M\right)
\prod_{i=1}^{M-2}\left(\cos^2\theta_i-\frac1{N-1}\sin^2\theta_i\right)\right.\\
&&+\sin^2\theta_{M-2}\left(1-\cos^2\theta_M\cos^2\theta_{M-1}\right)
\cdot\prod_{i=1}^{M-3}\left(\cos^2\theta_i-\frac1{N-1}\sin^2\theta_i\right)\\
&& +\cdots+\sin^2\theta_2\left(1-\prod_{i=3}^M\cos^2\theta_i\right)
\left(\cos^2\theta_1-\frac1{N-1}\sin^2\theta_1\right)\\
&& \left.+\sin^2\theta_1\left(1-\prod_{i=2}^M\cos^2\theta_i\right)\right].
\end{eqnarray*}

So $\sigma\le\frac1{\sqrt{N-1}}\Vert sin^2\theta_1,\sin^2\theta_2,\ldots,\sin\theta_{M-1}\Vert_2$.

\begin{conclusion}
\label{conc:SphereVectors} The cosine of angle $\theta$ of a chord made by $M$ chords of
given angles $\theta_1,\theta_2,\ldots,\theta_M$, between consecutive elements in a sequence of points on the unit sphere $S^{N-1}\subset E^{N}$, modeled by the above Markov chain, is with almost full probability, near%
\begin{equation}\label{eq:SphereVectors}
\prod_{i=1}^{M}\cos\theta_i.
\end{equation}
(With deviation $O\left((1/\sqrt{N})\cdot\Vert(\sin\theta_1,\sin\theta_2,\ldots,\sin\theta_{M-1})\Vert_2\right)$.)
\end{conclusion}

\section{What does (\ref{eq:SphereVectors}) tell us?}\label{sc:SphereVectors}
Firstly, we conclude, rather surprisingly, that if all $\theta_i$ were \emph{acute} angles -- between $0$ and $90^{\circ}$, i.e.\ with nonnegative cosines, then invariably \emph{also the accumulated angle is expected to be acute, no matter what $M$ -- the number of $\theta_i$'s is!}
\pa
Secondly, since multiplying by a cosine always does not increase the absolute value, the product in
(\ref{eq:SphereVectors}), giving the expected cosine of the accumulated angle, would tend to be small, meaning angle near $90^{\circ}$.
\pa
Moreover, if \emph{one} of the $\theta_1$ was $90^{\circ}$, the resulting expected $\theta$ is invariably $90^{\circ}$!
\pa
Speaking somewhat floridly: $90^{\circ}$ turns out to be both an impassible barrier and a forceful attractor!
\pa
Yet the wonder might subside if one recalls what we had found before: that anyhow it is highly unlikely for an angle to be far from $90^{\circ}$, thus `one should expect any `nudge' to push it into that dominating realm and if already there to remain there'.
\pa
Note also that for a tiny area on a big sphere, (\ref{eq:SphereVectors}) agrees with the `flat' case \S\ref{sc:Sum}: then the $\theta$'s are small, $\cos\theta\sim 1-\frac12\theta^2$ and multiplying these corresponds approximately to adding the $\theta^2$'s.

\section{As an Aside, Let's do the Markov Chain also for the `Flat' Case}
\label{sc:MarkovFlat}
We inquire about the norm of the sum $y$ of $M$ vectors $y_1,y_2,\ldots,y_M$ of known norms
$d_1,d_2,\ldots.d_M$ in the Euclidean $\mathbb{R}^{N}$. Their directions taken as distributed uniformly and independently.
\pa
Then, as before, the sum $y$ is the result of a Markov chain: Starting from $x_0=0$,
move to a point $x_1\in\Sigma(u_0,d_1)$ ($\Sigma(x,d)$ here denoting the sphere with center $x$ and
radius $d$), uniformly distributed there, then from that $x_1$ to a point $x_2\in\Sigma(x_1,d_2)$ uniformly distributed there, and so on, until one ends with $y=x_M$. And we wish to find
$\mathbb{E}\left[\Vert x_M-x_0\Vert^2\right]$.
\pa
As before, denote by ${\mathcal{L}}_d$ the operator mapping a function $p$ on $\mathbb{R}^{N}$ to the
function whose value at a vector $x\in\mathbb{R}^{N}$ is the average of $p$ on $\Sigma(u,\xi)$, then, in
our Markov chain, the expectation of $p(x_M)$ will be
\begin{equation}
\left({\mathcal{L}}_{d_M}{\mathcal{L}}_{d_{M-1}}\cdots{\mathcal{L}}_{d_1}(p)\right)(x_0).
\end{equation}
We are interested in
\begin{equation}
\mathbb{E}\left[\Vert x_M-x_0\Vert^2\right]  =
{\mathcal{L}}_{d_M}\mathcal{L}_{d_{M-1}}\cdots\mathcal{L}_{d_1}
(\Vert x-x_0\Vert^2)\Big\vert_{x=x_0},
\end{equation}
\pa
and will also be concerned with the standard deviation.
\pa
(This necessitates finding ${\mathcal{L}}_d(p)$ for $p$ polynomials of degree $\leq4$. While for general such $p$ the integrals of monomials over $S^{N-1}$ such as those calculated in appendix (\ref{sc:IntMon}) might be used, in our case these will not be needed.)
\pa
But note that for $a=(a_1,0,\ldots,0)$,\,\,
$\int\langle a,x\rangle^2\,d\omega_{N-1}=\int a_1^2x_1^2\,d\omega_{N-1}$,
i.e.,
$$\int\langle a,x\rangle^2\,d\omega_{N-1}=\frac1{N}a_1^2=\frac1{N}\Vert a\Vert^2,$$
which by symmetry holds for any $a$.
\pa
Now, the value of ${\mathcal{L}}_d(p)$ at $x_0$, the average of $p$ over $\Sigma(x_0,d)$ -- the sphere of radius $d$
around $x_0$, will be the average of $p(x-x_0)$ over $\Sigma(0,d)$ -- the sphere of radius $d$ around the origin.
\pa
In our case $p(x):=\Vert x\Vert^2$, and
$$p(x-x_0)=\Vert x\Vert^2+\Vert x_0\Vert^2-2\langle x_0,x\rangle.$$
whose average on $\Sigma(0,d)$ is $d^2+\Vert x_0\Vert^2$. So we have
$$\left({\mathcal{L}}_d(\Vert x\Vert^2)\right)\Big\vert_{x=x_0}=\Vert x_0\Vert^2+d^2.$$
Therefore
\begin{eqnarray*}
&& \mathbb{E}\left[\Vert x_M\Vert^2\right]=
{\mathcal{L}}_{d_M}\mathcal{L}_{d_{M-1}}\cdots\mathcal{L}_{d_1}
(\Vert x\Vert^2)\Big\vert_{x=0}=\\
&& \left(\Vert x\Vert^2+(d_1^2+d_2^2+\ldots d_M^2)\right)\Big\vert_{x=0}=
d_1^2+d_2^2+\ldots d_M^2.
\end{eqnarray*}
As for the standard deviation, it is given by
\begin{eqnarray*}
\sigma^2=&& \mathbb{E}\left[\Vert x_M\Vert^4\right]-\left(d_1^2+d_2^2+\ldots d_M^2\right)^2=\\
&& {\mathcal{L}}_{d_M}\mathcal{L}_{d_{M-1}}\cdots\mathcal{L}_{d_1}\left(\Vert x\Vert^4\right)\Big\vert_{x=0}
-\left(d_1^2+d_2^2+\ldots d_M^2\right)^2
\end{eqnarray*}
For $p(x):=\Vert x_M\Vert^4$ we have
\begin{eqnarray*}
&& p(x-x_0)=\left(\Vert x\Vert^2+\Vert x_0\Vert^2-2\langle x_0,x\rangle\right)^2=\\
&& \Vert x\Vert^4+\Vert x_0\Vert^4+4\langle x_0,x\rangle^2+2\Vert x_0\Vert^2\cdot\Vert x\Vert^2-
4\Vert x_0\Vert^2\langle x_0,x\rangle-4\Vert x\Vert^2\cdot\langle x_0,x\rangle.
\end{eqnarray*}
Its average on $\Sigma(0,d)$, which is $\left({\mathcal{L}}_d(p)\right)(x_0)$, will be
\begin{eqnarray*}
&& d^4+\Vert x_0\Vert^4+4\frac1{N}d^2\Vert x_0\Vert^2+2d^2\Vert x_0\Vert^2\\
&& = \Vert x_0\Vert^4+\left(2+4\frac1{N}\right)d^2\Vert x_0\Vert^2+d^4.
\end{eqnarray*}
Thus,
\begin{eqnarray*}
&& \Vert x_0\Vert^4\to_{{\mathcal{L}}_{d_1}}d_1^4
+\left(2+4\frac1{N}\right)d_1^2\Vert x\Vert^2+\Vert x\Vert^4\\
&& \to_{{\mathcal{L}}_{d_2}}d_1^4+d_2^4
+\left(2+4\frac1{N}\right)d_1^2\left(d_2^2+\Vert x\Vert^2\right)
+\left(2+4\frac1{N}\right)d_2^2\Vert x\Vert^2+\Vert x\Vert^4\\
&& \to_{{\mathcal{L}}_{d_3}}d_1^4+d_2^4+d_3^4
+\left(2+4\frac1{N}\right)d_1^2\left(d_2^2+d_3^2+\Vert x\Vert^2\right)\\
&& +\left(2+4\frac1{N}\right)d_2^2\left(d_3^2+\Vert x\Vert^2\right)
+\left(2+4\frac1{N}\right)d_3^2\Vert x\Vert^2+\Vert x\Vert^4
\to_{{\mathcal{L}}_{d_3}}\cdots\\
&& \to_{{\mathcal{L}}_{d_M}}d_1^4+d_2^4+\cdots
+d_M^4+\left(2+4\frac1{N}\right)\sum_{i<j}d_i^2d_j^2\\
&& +\left(2+4\frac1{N}\right)\left(d_1^2+\cdots+d_M^2\right)\Vert x\Vert^2+\Vert x\Vert^4.
\end{eqnarray*}
And
\begin{eqnarray*}
&& \sigma^2=\mathbb{E}\left[\Vert x_M\Vert^4\right]-\left(d_1^2+d_2^2+\ldots d_M^2\right)^2=\\
&& {\mathcal{L}}_{d_M}\mathcal{L}_{d_{M-1}}\cdots\mathcal{L}_{d_1}\left(\Vert x\Vert^2\right)\Big\vert_{x=0}
-\left(d_1^2+d_2^2+\ldots+d_M^2\right)^2=\\
&& -\left(d_1^2+d_2^2+\ldots+d_M^2\right)^2+
d_1^4+d_2^4+\cdots+d_M^4+\left(2+4\frac1{N}\right)\sum_{i<j}d_i^2d_j^2\\
&& =4\frac1{N}\sum_{i<j}d_i^2d_j^2.
\end{eqnarray*}
Thus, for any $\alpha\ge0$,
\begin{eqnarray*}
&& \sigma^2=4\frac1{N}\sum_{i<j}d_i^2d_j^2\le 4\frac1{N}\sum_{i<j}\left(d_i^2d_j^2+\alpha\left(d_i^2-d_j^2\right)^2\right)=\\
&& 4\frac{N-1}{N}\alpha\left(d_1^2+d_2^2+\ldots+d_N^2\right)^2-8\frac{N-1}{N}\alpha\sum_{i<j}d_i^2d_j^2
+4\frac1{N}\left(1-2\alpha\right)\sum_{i<j}d_i^2d_j^2=\\
&& 4\frac{N-1}{N}\alpha\left(d_1^2+d_2^2+\ldots+d_N^2\right)^2
+\left(\frac4{N}-8\alpha\right)\sum_{i<j}d_i^2d_j^2
\end{eqnarray*}
So take $\alpha=\frac1{2N}$ which will annul the last term, to find
$$\sigma\le\frac{\sqrt{2(N-1)}}{N}\left(d_1^2+d_2^2+\ldots+d_N^2\right),$$
in accordance with Conclusion \ref{conc:SpaceVectors}.

\section{The action of a `Random' Symmetric linear operator With Given Eigenvalues in a High-Dimensional
Space\label{sc:LinOp}}

Consider an $N\times N$ symmetric matrix $A$ with given eigenvalues $s_1,s_2,\ldots,s_N$
. This means that
$$A=U^{-1}\operatorname*{diag}(s_1,s_2,\ldots,s_N)U$$
where $U$ is orthogonal, which we take random with uniform distribution.
\pa
Let $A$ act on a fixed unit vector $u$. One gets a distribution for $Tu$. Note that for a fixed orthogonal $V$ which
fixes $u$, replacing the random $U$ by $UV^{-1}$ does not change the distribution, but replaces $Tu$ by $VTu$.
Therefore \textbf{the distribution of $Tu$ is invariant under action by any orthogonal $V$ which fixes $u$}.
\pa
This means that that distribution will be determined if we know the distribution of $\Vert Tu\Vert$ and
$\cos\angle(Tu,u)=\langle Tu, u\rangle/\Vert Tu\Vert\Vert u\Vert$.
\pa
As for $\Vert Tu\Vert$, it is $\Vert U'SUu\Vert$ where
$S:=\operatorname*{diag}(s_1,s_2,\ldots,s_N)$,\,\,
$Uu$ uniformly distributed on the unit sphere.
\pa
By Section \ref{sc:ConcMeas}, that would be almost as $S$ applied to
$(1/\sqrt{N})x$,\, \, $x$ with coordinates i.i.d.\ $\sim
{\mathcal{N}}$, which is, of course, a vector with independent coordinates but
the $j$-th coordinate distributed as $(1/\sqrt{N})s_j$ times ${\mathcal{N}}$.
\pa
And, similarly to what we had in Section \ref{sc:ConcMeas}, the square of
the norm of $S\cdot(1/\sqrt{N})x$, which is $(1/N)\sum_{j=1}^{N}s_{j}^2%
x_{j}^2$ has mean
\begin{equation}
(1/N)\sum_{j=1}^{N}s_j^2=\left(\Vert(s_1,s_2,\ldots,s_N)\Vert
_2^{(\pi)}\right)^2,
\end{equation}
around which it is concentrated -- its standard deviation being
\begin{equation}
\sigma_0\cdot\sqrt{(1/N^2)\sum_{j=1}^{N}s_{j}^{4}}=(1/\sqrt{N})\sigma_0
\cdot\left(\Vert(s_1,s_2,\ldots,s_N)\Vert_{4}^{(\pi)}\right)^2,
\end{equation}
where $\sigma_0$ is the standard deviation for $x^2$ when $x\sim{\mathcal{N}}%
$, namely,
\begin{equation}
\sigma_0=\sqrt{\frac1{\sqrt{2\pi}}\int(x^2-1)^2\operatorname{exp}%
(-\textstyle{\frac12}x^2)\,dx}=\sqrt2.
\end{equation}
By \ref{sc:Norm}, the relative deviation is, thus, expected, with
almost full probability, to be $O(1/\sqrt{N})$.\pa

Note that since $A=U^{-1}\cdot\operatorname*{diag}(s_1,s_2,\ldots,s_N)\cdot U$, the value
around which that norm $\Vert Tu\Vert$ is concentrated would be
\begin{equation}
\Vert(s_1,s_2,\ldots,s_N)\Vert_2^{(\pi)}=
(1/\sqrt{N})\Vert S\Vert_{HS}=(1/\sqrt{N})\Vert A\Vert_{HS}.
\end{equation}

As for $\langle Tu,u\rangle$, it is $\langle SUu,Uu\rangle$, $Uu$ distributed uniformly.
And, as we did above, replace $Uu$ by $(1/\sqrt{N})x$,\, \, $x$ with coordinates i.i.d.\
$\sim{\mathcal{N}}$,

\begin{equation}
(1/N)\langle Sx,x\rangle=(1/N)\sum_{j=1}^{N}s_jx_j^2, \label{eq:S}%
\end{equation}
which has mean $(1/N)\sum_{j=1}^{N}s_j=(1/N)\operatorname*{tr}\,A$ and
$(1/\sqrt{N})\sigma\Vert(s_1,\ldots,s_N)\Vert_2^{(\pi)}$ is its standard
deviation. Of course, if $A$ is positive semidefinite then the $s_{\ell}\geq0$
and the above mean is $\Vert(s_1,s_2,\ldots,s_N)\Vert_1^{(\pi)}$. This
leads to the following conclusion.

\begin{conclusion}\label{conc:ActMat}
An $N\times N$ symmetric matrix $A$ with given eigenvalues
$s_1,s_2,\ldots,s_N$, acting on a high-dimensional $E^{N}$, would be
expected to multiply the norm of a fixed vector $u$, with almost full probability, by
\begin{equation}
\Vert(s_1,s_2,\ldots,s_N)\Vert_2^{(\pi)}=(1/\sqrt{N})\Vert A\Vert_{HS},%
\end{equation}
up to a relative deviation $O(1/\sqrt{N})$,
while the cosine of the angle between $v$ and $Tv$ is, with almost full probability,
near (with deviation $O(1/\sqrt{N})$)%
\begin{equation}
\dfrac{(1/N)\operatorname*{tr}\,A}{(1/\sqrt{N})\Vert A\Vert_{HS}}
=\dfrac{(1/N)\operatorname*{tr}\,A}{\Vert(s_1,s_2,\ldots,s_N)\Vert_2^{(\pi)}},
\end{equation}
which, if $A$ is positive-semidefinite, is equal to%
\begin{equation}
\dfrac{\Vert(s_1,s_2,\ldots,s_{N})\Vert_1^{(\pi)}}%
{\Vert(s_1,s_2,\ldots,s_{N})\Vert_2^{(\pi)}}.
\end{equation}
Otherwise the distribution of $Au$ is invariant w.r.t.\ rotations in the hyperplane orthogonal to $u$.
\end{conclusion}

For a product $A_{M}A_{M-1}\cdots A_1$, of a sequence of \emph{symmetric} operators $A_i$
with given eigenvalues $(s_1^{(i)},s_2^{(i)},\ldots,s_N^{(i)})$\,\,
$A=U^{\prime}\operatorname*{diag}(s_1,s_2,\ldots,s_{N})U$ with
the $U$ independently uniformly distributed on the orthogonal group,
employ the Markov chain in \S\ref{sc:Sphere}, with independent
uniform distributions on the spheres, and one has by Conclusion \ref{conc:SphereVectors},

\begin{conclusion}
\label{conc:ProdMat} For a product $A_{M}A_{M-1}\cdots A_1$, of a sequence
of \emph{symmetric} operators $A_{i}$ with given eigenvalues
$(s_1^{(i)},s_2^{(i)},\ldots,s_N^{(i)})$, the cosine of the angle between
$v$ and $A_{M}A_{M-1}\cdots A_1v$, is, with almost full probability, near
(with deviation $O(\sqrt{M}/\sqrt{N})$)%
\begin{equation}
\prod_{i=1}^{M}\dfrac{(1/N)\operatorname*{tr}\,A_{i}}{(1/\sqrt
{N})\Vert A_{i}\Vert_{HS}}
  =\prod_{i=1}^{M}\dfrac{(1/N)\operatorname*{tr}A_{i}}{\Vert(s_1^{(i)},s_2^{(i)},\ldots
,s_{N}^{(i)})\Vert_2^{(\pi)}},
\end{equation}
which, if for all $i,$ $A_i$ is positive semidefinite, is equal to%
\begin{equation}
\prod_{i=1}^{M}\dfrac{\Vert(s_1^{(i)},s_2^{(i)},\ldots,s_N^{(i)})\Vert_1^{(\pi)}}
{\Vert(s_1^{(i)},s_2^{(i)},\ldots,s_N^{(i)})\Vert_2^{(\pi)}}  .\label{eq:ProdMat}%
\end{equation}
\end{conclusion}

\begin{remark}
\label{remark:ProdMat}
Note that if the $A_i$ are positive semidefinite, the value (\ref{eq:ProdMat}) around which the square of the
cosine is concentrated, is nonnegative, that is, the angle between the vectors is $\leq90^{\circ}$.
\bigskip
\end{remark}

\section{Recap: The `Kappa'-Calculations in the Superiorization Article \cite{cl-2020}-- Somewhat `Neater' Formulas}
For the accumulated terms $({\mathbf1}+d\cdot{\mathbf{\kappa}})^{-1}$ along
the path (in what follows we denote by $k$ indices along the path, i.e.,
$k\in\operatorname*{path})$ denote the eigenvalues (here, also the singular
values) of the encountered ${\mathbf{\kappa}}_{k}$ (the curvature operator in
the hyperplane $H$), i.e.,\ the relevant principal curvatures, by
$(\kappa_{\ell}^{(k)})_{\ell},$ for $\ell=1,2,\ldots,N-1$. Then those of
$({\mathbf1}+d\cdot{\mathbf{\kappa}})^{-1}$ are $\left(  (1+d_{k}\cdot
\kappa_{\ell}^{(k)})^{-1}\right)  _{\ell}$, so that, by Conclusion
\ref{conc:ActMat}, and using the $\Vert\cdot\,\Vert_{p}^{(\pi)}$ norm of
Appendix \ref{sc:Norm} below, for $(N-1)$-dimensional vectors, their product
is expected to multiply the norm of the vector they act upon by
\begin{equation}
\prod_{k\in\operatorname*{path}}
\left\Vert\left((1+d_{k}\cdot\kappa_{\ell}^{(k)})^{-1}\right)_{\ell}\right\Vert_2^{(\pi)},\label{eq:Norm}%
\end{equation}
still with relative deviation of the order of at most $O(\sqrt{n-i}/\sqrt{N})$.\pa

By Conclusion \ref{conc:ProdMat}, they are
expected to rotate the direction of the vector, i.e.,\ shift the normalized
vector, by an angle with cosine
\begin{equation}
\prod_{k\in\operatorname*{path}}
\dfrac{\left\Vert \left((1+d_{k}\cdot\kappa_{\ell}^{(k)})^{-1}\right)_{\ell}
\right\Vert_1^{(\pi)}}{\left\Vert \left((1+d_{k}\cdot\kappa_{\ell}^{(k)})^{-1}\right)_{\ell}\right\Vert_2^{(\pi)}}. \label{eq:Devi}%
\end{equation}
with relative deviation of the order of at most $O(\sqrt{n-i}/\sqrt{N})$.\pa

Observe that $\Vert\,\cdot\Vert_1^{(\pi)}$ $\leq$
$\Vert\,\cdot\Vert_2^{(\pi)}$ (cf.\ Appendix \ref{sc:Norm}) and, by\ Remark
\ref{remark:ProdMat}, the value of (\ref{eq:Devi}) is always $\leq1$,
meaning angle of rotation $\leq90^{\circ}$. Indeed, in many cases it will be
much less than $90^{\circ}$. For example, for vectors
$\left((1+d_{k}\cdot\kappa_{\ell}^{(k)})^{-1}\right)_{\ell}$ with equal
(resp.\ almost equal) entries (in our case -- either `spherical'\ curvature or when the
$d_{k}\cdot\kappa$ are small), the $\Vert\cdot\,\Vert_2^{(\pi)}$ norm will
be equal (resp.\ almost equal) to the $\Vert\cdot\,\Vert_1^{(\pi)}$ norm,
hence the terms in the product in (\ref{eq:Devi}) will be near $1$.\pa

Both (\ref{eq:Norm}) and (\ref{eq:Devi}) refer to the $(N-1)$-dimensional
vectors $v=((1+d_{k}\cdot\kappa_{\ell}^{(k)})^{-1})_{\ell}$, having entries in
$(0,1]$. In (\ref{eq:Norm}), which controls how much the norm was reduced, we
have the product of $\Vert v\Vert_2^{(\pi)}$. In (\ref{eq:Devi}), which
controls how much the direction was rotated, we have the product of
$\Vert v\Vert_1^{(\pi)}/\Vert v\Vert_2^{(\pi)}$.

\begin{proposition}
For an $(N-1)$-dimensional vector $v=(v_{\ell})_{\ell}$ with components
$v_{\ell}\in(0,1]$, we have
\begin{equation}
\left(\Vert v\Vert_2^{(\pi)}\right)^2\leq\Vert v\Vert_1^{(\pi)}%
\leq\frac12\left(\left(\Vert v\Vert_2^{(\pi)}\right)^2+1\right).
\end{equation}
\end{proposition}

\begin{proof}
Since $v_{\ell}\in(0,1]$, one has $v_{\ell}^2\leq v_{\ell}$. Averaging, we
get $\left(  \Vert v\Vert_2^{(\pi)}\right)  ^2\leq\Vert v\Vert_1^{(\pi
)}$. Also, by definition of $\Vert v\Vert_2^{(\pi)}$ for $(N-1)$-dimensional
vectors, see Appendix \ref{sc:Norm},
\begin{eqnarray*}
&& \left(\Vert v\Vert_2^{(\pi)}\right)^2=\dfrac1{N-1}\sum_{\ell=1}^{N-1}v_{\ell}^2
=1-\dfrac1{N-1}\sum_{\ell=1}^{N-1}(1-v_{\ell}^2)\\
&& =1-\dfrac1{N-1}\sum_{\ell=1}^{N-1}(1-v_{\ell})(1+v_{\ell})\geq1-
2\dfrac1{N-1}\sum_{\ell=1}^{N-1}(1-v_{\ell})\\
&& =2\dfrac1{N-1}\sum_{\ell=1}^{N-1}v_{\ell}-1=2\Vert v\Vert_1^{(\pi)}-1.
\end{eqnarray*}
Hence, $\Vert v\Vert_1^{(\pi)}\leq\frac12(\Vert v\Vert_2^{(\pi)})^2+1)$,
which completes the proof.
\end{proof}

As a consequence of this proposition we have,
\begin{eqnarray*}
&& \dfrac{\Vert v\Vert_1^{(\pi)}}{\Vert v\Vert_2^{(\pi)}}\geq
\dfrac{(\Vert v\Vert_2^{(\pi)})^2}{\Vert v\Vert_2^{(\pi)}}=\Vert v\Vert_2^{(\pi)},\\
&& \dfrac{\Vert v\Vert_1^{(\pi)}}{\Vert v\Vert_2^{(\pi)}}\leq
\frac12\dfrac{(\Vert v\Vert_2^{(\pi)})^2+1}{\Vert v\Vert_2^{(\pi)}}=
\frac12\left(\Vert v\Vert_2^{(\pi)}+\dfrac1{\Vert v\Vert_2^{(\pi)}}\right),
\end{eqnarray*}

So, there is here a \emph{`balancing effect'} -- if the angle of rotation becomes close to $90^{\circ}$
in (\ref{eq:Devi}), then the norm will be reduced considerably in (\ref{eq:Norm}). Thus, when $i$ is such
that $d_{i}$ times a `typical' curvature ${\mathbf{\kappa}}$ (loosely, the ratio between $d$ and a
` typical'\ radius of the $C_{i}$) is still considerably larger than $1$ (maybe while in the early columns
of the superiorization matrix with $i$ small), then, by (\ref{eq:Norm}), the cascade of $DP$ will reduce the norm hugely, hence, anyway applying $\nabla\phi$ then will give a negligible result.\pa

On the other hand, when we reach a stage where $d_{i},d_{i+1},\ldots,d_{n}$ are small, both the possible rotation
and the distance traveled are controlled. But of course, then the decrease of the $\beta_{k}$ should also be
taken into account. For big $i$, thus small $\beta_{i}$, the contribution might again be negligible. This shows that the main contribution in (8) of \cite{cl-2020} seems to come from intermediate terms.\pa

As said above, the angle of rotation, both by the $\alpha$ and by the $\kappa$
seems to be controlled, as long as the number of steps $n$ does not approach
the vector space dimension $N$. If conditions are imposed on the target
function $\phi$ then point (3) above could also be tackled, in view of the
preceeding paragraph, bringing our analysis closer to conclusion.

\appendix
\section{The probability $L^{p}$ norms of vectors\label{sc:Norm}}
For a vector $x\in E^{N}$, and $1\leq p<\infty$, denote by $\Vert\,\cdot\Vert_{p}^{(\pi)}$ ($\pi$ stands
for ` probability space') its $L^{p}$ norm when the set of indices $\{1,2,\ldots,N\}$ is made into a uniform probability space, giving each index a weight $1/N$, namely
\begin{equation}
\Vert x\Vert_{p}^{(\pi)}:=\left(\dfrac1{N}\sum_{j=1}^{N}|x_{j}|^{p}\right)^{1/p}, \label{eq:Lp}%
\end{equation}
see, e.g., \cite{song97}. As with any probability measure, always $\Vert\,\cdot\Vert_{p}^{(\pi)}$ increases with $p$.
\pa
For $x_1,x_2,\ldots,x_{N}$ i.i.d.\ $\sim{\mathcal{N}}$, $\left(\Vert x\Vert_{p}^{(\pi)}\right)^{p}$ is an average:
its expectation ${\mathbb{E}}$ will be the same as the expectation of $|x|^{p}$ for $x$ a scalar distributed
$\sim{\mathcal{N}}$:
\begin{equation}
{\mathbb{E}}\left[|x|^{p}\right]=\dfrac1{\sqrt{2\pi}}\int|x|^{p}%
\operatorname{exp}(-\textstyle{\frac12}x^2)\,dx,
\end{equation}
but its standard deviation will be $1/\sqrt{N}$ that of $|x|^{p}$ for a scalar
$\sim{\mathcal{N}}$:
\begin{equation}
\dfrac1{\sqrt{N}}\dfrac1{\sqrt{2\pi}}
\int\left(|x|^{p}-\mathbb{E}\left[|y|^{p}\right]\right)^2
\operatorname{exp}(-\textstyle{\frac12}x^2)\,dx.
\end{equation}
Thus, $\Vert x\Vert_{p}^{(\pi)}$ is highly concentrated around the, not
depending on $N$, $\left({\mathbb{E}}\left[|x|^{p}\right]\right)^{1/p}$
with degree of concentration $O(1/\sqrt{N})$.

One may conclude, loosely speaking, that in any case, these $\Vert\,\cdot
\Vert_{p}^{(\pi)}$ norms, having not depending on $N$ means, are expected to
be $O(1)$, for all $N$.

\section{The integrals over $S^{N-1}\subset\mathbb{R}^{N}$ of monomials of degree $\leq4$ of $x_1,x_2,\ldots,x_N$} \label{sc:IntMon}
Clearly these are $0$ if some power of some $x_i$ is odd. The integral of $x_i^2$ is $\frac1{N}$ (\ref{eq:integral}). So we are left with $\int x_i^4$ and $\int x_i^2x_j^2$,\,$i\ne j$, both, by symmetry, the same for all relevant $i$ and $j$.
\pa
To calculate these, note that
\begin{eqnarray*}
&& 1=\int(x_1^2+\cdots+x_N^2)^2\,d\omega_{N-1}=N\int x_1^4+N(N-1)\int x_1^2x_2^2\quad\textrm{(*)}.\\
&& \textrm{A rotation by }45^{\circ}\textrm{ in the }x_1,x_2\textrm{ plane transforms}\\
&& x_1\to(1/\sqrt2)(x_1+x_2)\textrm{ and }x_2\to(1/\sqrt2)(x_1-x_2),\\
&& \textrm{ hence }x_1x_2\to\frac12(x_1^2-x_2^2)\textrm{ thus }
x_1^2x_2^2\to\frac14(x_1^2-x_2^2)^2.\textrm{ Therefore}\\
&& \int x_1^2x_2^2=\frac14\int(x_1^2-x_2^2)^2=\frac14\left(2\int x_1^4-2\int x_1^2x_2^2\right).\\
&& \textrm{Consequently }\int x_1^4=3\int x_1^2x_2^2.\,\,\textrm{And finally, by (*)},\\
&& \int x_1^2x_2^2\,d\omega_{N-1}=\frac1{N(N+2)},\qquad\int x_1^4\,d\omega_{N-1}=\frac{3}{N(N+2)}.
\end{eqnarray*}

\section{In the Hyperbolic Space}
Having discussed the sphere, it would be illuminating to do the same for the \emph{hyperbolic space}, proceeding with much analogy.
\pa
Model the $N-1$-dimensional hyperbolic space as the subset ${\mathbf{H}}^{N-1}$, of an $N$-dimensional `Minkowski' space with signature $(+,-,\cdots,-)$, consisting of `future' vectors (in the sense of Relativity Theory) with norm $1$, i.e.\ the half of the hyperboloid $x_1^2-x_2^2-\ldots -x_N^2=1$ that lies in the half-space $x_1>0$.
\pa
Denote by $\langle\,,\,\rangle_\mathbf{H}$ the `Minkowski' inner product
$\langle x,y \rangle_{\mathbf{H}}:=x_1y_1-x_2y_2-\ldots-x_Ny_N$, thus
$\Vert x\Vert^2_\mathbf{H}=\langle x,x \rangle_{\mathbf{H}}=x_1^2-x_2^2-\ldots -x_N^2$,
while keeping the notation $\langle x,y \rangle=x_1y_1+x_2y_2+\ldots+x_Ny_N$.
\pa
The `distance' between vectors $u,v\in{\mathbf{H}}^{N-1}$, i.e.\ the distance along a straight line -- a geodesic, is the \emph{`hyperbolic arc'} $\xi$ defined by $\cosh\xi=\langle u,v\rangle_\mathbf{H}$.
\pa
And as above we wish to find what should we expect the accumulated hy-arc of a path made of $M$ steps of given hy-arcs, $\xi_1,\xi_2,\ldots.\xi_M$, to be.
\pa
Here again, for some fixed $\xi\geq 0$, the set of points in ${\mathbf{H}}^{N-1}$ of
hy-arc $\xi$ from some fixed vector $u\in{\mathbf{H}}^{N-1}$ is an $(N-2)$-sphere
$\subset{\mathbf{H}}^{N-1},$ $\Sigma(u,\xi)$ given by
\begin{equation}
\Sigma(u,\xi):=\cosh\xi\cdot u+\sinh\xi\cdot{S^{N-2}}_{u^{\bot}},\label{eq:hy}%
\end{equation}
where ${S^{N-2}}_{u^{\bot}}$ stands for the unit sphere in the hyperplane perpendicular to $u.$
(on which the induced  geometry is negative-squared-norm Euclidean, not hyperbolic!)
\pa
Again, one performs a Markov chain. Start from a point $u_0$ on $\mathbf{H}^{N-1}$, move to a point $u_1\in\Sigma(u_0,\xi_1)$ uniformly distributed there, then from that $u_1$ to a point
$u_2\in\Sigma(u_1,\xi_2)$ uniformly distributed there, and so on, until one ends with $u_M$.
We would like to find the expected hyperbolic cosine of the hy-arc between $u_0$ and
$u_M$, namely $\mathbb{E}\left[\langle u_M,u_0\rangle_{\mathbf{H}}\right]$.
\pa
And again, denote by ${\mathcal{L}}_\xi$ the operator mapping a function $p$ on
$S^{N-1}$ to the function whose value at a vector $u\in\mathbf{H}^{N-1}$ is the
average of $p$ on $\Sigma(u,\xi)$, then ${\mathcal{L}}_{\xi_k}(p)$ evaluated at
$u$ is the expectation of $p$ at the point to which $u$ moved in the $k$-th
step above. Hence, in the above Markov chain, the expectation of $p(u_M)$
is
\begin{equation}
\left({\mathcal{L}}_{\xi_M}{\mathcal{L}}_{\xi_{M-1}}\cdots{\mathcal{L}}_{\xi_1}(p)\right)(u_0).
\end{equation}
Thus, what we are interested in is
\begin{equation}
\mathbb{E}\left[\langle u_M,u_{0}\rangle_{\mathbf{H}}\right]  =
{\mathcal{L}}_{\xi_M}{\mathcal{L}}_{\xi_{M-1}}\cdots{\mathcal{L}}_{\xi_1}(\langle x,u_0\rangle_\mathbf{H})\Big\vert_{x=u_0}.
\end{equation}
\pa
So, let us calculate ${\mathcal{L}}_{\xi}(p)$ for polynomials of degree $\leq2$
like in (\ref{eq:poly}).
\pa
\textbf{Care in working with an invariant trace}\par
Here we seek invariance w.r.t.\ the \emph{`Lorentz group'} -- linear transformations that preserve the `Minkowski' norm -- so we must be careful in speaking about trace when
referring to a quadratic form.
\pa
Indeed, let
$$p(x)=\langle Qx,x\rangle+2\langle a,x\rangle_\mathbf{H}+\gamma.$$
An invariant trace $\operatorname*{tr}_{\mathbf{H}}Q$ here would be a linear mapping from the bilinear forms to the scalars, which maps a bilinear form of rank $1$
\begin{equation} \label{eq:BilFrm}
(x,y)\mapsto\langle x,a\rangle_{\mathbf{H}}\cdot\langle y,b\rangle_{\mathbf{H}}
\end{equation}
to $\langle a,b\rangle_\mathbf{H}$.
\pa
In trying to write (\ref{eq:BilFrm}) as $\langle Qx,y\rangle$, denote
$$\tilde{x}:=(x_2,x_3,\ldots,x_N)\in E^{N-1}.$$
then $x=(x_1,\tilde{x})$, and $\langle x,y\rangle=x_1y_1-\langle\tilde{x},\tilde{y}\rangle.$
Our bilinear form (\ref{eq:BilFrm}) will be
$$(x,y)\mapsto \left(a_1x_1-\langle\tilde{a},\tilde{x}\rangle\right)\cdot
\left(b_1y_1-\langle\tilde{a},\tilde{y}\rangle\right)=\langle Qx,y\rangle$$
where
\begin{equation}
Q=\left(
\begin{array}
[c]{cc}%
a_1b_1 & -b_1\tilde{a}^{\prime}\\
-a_1\tilde{b} & \tilde{b}\cdot\tilde{a}^{\prime}%
\end{array}
\right).
\end{equation}
And our $\operatorname*{tr}_{\mathbf{H}}$ should map it to
$$\langle a,b\rangle_{\mathbf{H}}=a_1b_1-\langle\tilde{a},\tilde{b}\rangle$$
Which mandates
\begin{equation}\label{eq:hyTrace}
{\operatorname*{tr}}_{\mathbf{H}}%
\left(
\begin{array}
[c]{cc}%
q_{11} & c^{\prime}\\
d & Q_1%
\end{array}
\right)%
:=q_{11}-{\operatorname*{tr}}Q_1.
\end{equation}

Returning to calculating ${\mathcal{L}}_{\xi}(p)$ for
$$p(x)=\langle Qx,x\rangle+2\langle a,x\rangle_{\mathbf{H}}+\gamma,$$
i.e.\ its value at some $u\in\mathbf{H}^{N-1}$, we may again, in performing the calculation,
assume $u=(1,0,\ldots,0)$, and $Q$ as in (\ref{eq:hyTrace}).
\pa
Note that for our $u=(1,0,\ldots,0)$,\,\,\,
$a_1=\langle a,u\rangle_{\mathbf{H}}$,\,\,\,
$q_{11}=\langle Qu,u\rangle$ and
${\operatorname*{tr}}\,Q_1=q_{11}-{\operatorname*{tr}}_{\mathbf{H}}Q=
\langle Qu,u\rangle-{\operatorname*{tr}}_{\mathbf{H}}Q$. And we have
\begin{eqnarray*}
&& p(x)=\langle Qx,x\rangle+2\langle a,x\rangle_{\mathbf{H}}+\gamma\\
&& = q_{11}x_1^2+\langle c+d,\tilde{x}\rangle+\langle Q_1\tilde{x},\tilde{x}\rangle
+2a_1x_1-2\langle\tilde{a},\tilde{x}\rangle+\gamma
\end{eqnarray*}
Hence, taking account of (\ref{eq:hy}) for $u=(1,0,\ldots,0)$, and using
(\ref{eq:integral}),
\begin{eqnarray*}
&& \left({\mathcal{L}}_\xi p\right)(u)=\left({\mathcal{L}}_\xi p\right)(1,0,\ldots,0)=\\
&& =\cosh^2\xi\cdot q_{11}+\frac1{N-1}\sinh^2\xi\cdot{\operatorname*{tr}}\,Q_1+
2\cosh\xi\cdot a_1+\gamma=\\
&& =\cosh^2\xi\cdot\langle Qu,u\rangle+
\frac1{N-1}\sinh^2\xi\cdot\left(\langle Qu,u-{\operatorname*{tr}}_{\mathbf{H}}Q\rangle\right)
 +2\cosh\xi\cdot\langle a,u\rangle_{\mathbf{H}}+\gamma\\
&& =\left(\cosh^2\xi+\frac1{N-1}\sinh^2\xi\cdot\right)\langle Qu,u\rangle
 -\frac1{N-1}\sinh^2\xi\cdot{\operatorname*{tr}}_{\mathbf{H}}Q+
 2\cosh\xi\cdot\langle a,u\rangle_{\mathbf{H}}+\gamma.
\end{eqnarray*}
which, by symmetry, will hold for any $u\in{\mathbf{H}}^{N-1}$.
\pa
We are interested, for some fixed $u\in{\mathbf{H}}^{N-1}$, in
\begin{equation}
p(x)=\langle u,x\rangle_\mathbf{H}.
\end{equation}
Then there is no $Q$ term, so one has
\begin{equation}
\Big({\mathcal{L}}_\xi\big(\langle u,x\rangle_{\mathbf{H}}\big)\Big)(u)=
\cosh\xi\cdot\langle u,x\rangle_{\mathbf{H}}.
\end{equation}
Consequently,
\begin{eqnarray*}
&&  {\mathbb{E}}\left[\langle u_M,u_0\rangle_{\mathbf{H}}\right] =
\left({\mathcal{L}}_{\xi_{M}}{\mathcal{L}}_{\xi_{M-1}}\cdots{\mathcal{L}}_{\xi_1}
(\langle x,u_0\rangle)_\mathbf{H}\right)\vert_{x=u_0}=\\
&& =\prod_{i=1}^M\cosh\xi_i\cdot(\langle u_0,x\rangle_{\mathbf{H}})\Big\vert_{x=u_0}=
\prod_{i=1}^{M}\cosh\xi_i.
\end{eqnarray*}
\pa
And we also assess the standard deviation, which is
\begin{equation}
\sigma=\sqrt{\left({\mathcal{L}}_{\xi_{M}}{\mathcal{L}}_{\xi_{M-1}}\cdots
{\mathcal{L}}_{\xi_1}\right)(\langle u_0,x\rangle_\mathbf{H}^2)\Big\vert_{x=u_0}
-\left(\prod_{i=1}^M\cosh\xi_i\,\right)^2}.
\end{equation}
Here $p(x)$ is of the form $\langle c,x\rangle_\mathbf{H}^2$, and by symmetry we assume
$c=(c_1,0,\ldots,0)$ making $p(x)=c_1^2x_1^2$. So there is only the $Q$ term with
$\langle Qx,x\rangle:=c_1^2x_1^2$. Then
${\operatorname*{tr}}_{\mathbf{H}}\,Q=c_1^2=\Vert c\Vert^2_\mathbf{H}$,
which by symmetry holds for general $c$. And we find
\begin{eqnarray*}
&& \Big({\mathcal{L}}_{\xi}\big(\langle c,x\rangle_\mathbf{H}^2\big)\Big)(u)=\\
&& \left(\cosh^2\xi+\frac1{N-1}\sinh^2\xi\right)\langle c,u\rangle_\mathbf{H}^2-
\frac1{N-1}\sinh^2\xi\cdot\Vert c\Vert^2_\mathbf{H}.
\end{eqnarray*}
Consequently, for $c=u_0$ (note $\Vert u_0\Vert^2_\mathbf{H}=1$),
\begin{eqnarray*}
&& \sigma^2=\left({\mathcal{L}}_{\xi_M}{\mathcal{L}}_{\xi_{M-1}}\cdots
{\mathcal{L}}_{\xi_1}\right)(\langle u_0,x\rangle^2)\Big\vert_{x=u_0}-
\left(\prod_{i=1}^M\cosh\xi_i\right)^2=\\
&& =-\left(\prod_{i=1}^M\cosh\xi_i\right)^2+
\prod_{i=1}^M\left(\cosh^2\xi_i+\frac1{N-1}\sinh^2\xi_i\right)\\
&& -\frac1{N-1}\left[\sinh^2\xi_1
+\sinh^2\xi_2\left(\cosh^2\xi_1+\frac1{N-1}\sinh^2\xi_1\right)\right.\\
&& +\sinh^2\xi_3\left(\cosh^2\xi_2+\frac1{N-1}\sinh^2\xi_2\right)
\left(\cosh^2\xi_1+\frac1{N-1}\sinh^2\xi_1\right)\\
&& \left.+\cdots+\sinh^2\xi_M\cdot\prod_{i=1}^{M-1}\left(\cosh^2\xi_i+\frac1{N-1}\sinh^2\xi_i\right)\right],
\end{eqnarray*}
\begin{eqnarray*}
&&=-\left(\prod_{i=1}^M\cosh\xi_i\right)^2
+\prod_{i=1}^M\left(1+\frac{N}{N-1}\sinh^2\xi_i\right)\\
&& -\frac1{N-1}\left[\sinh^2\xi_1
+\sinh^2\xi_2\left(1+\frac{N}{N-1}\sinh^2\xi_1\right)\right.\\
&& +\sinh^2\xi_3\left(1+\frac{N}{N-1}\sinh^2\xi_2\right)
\left(1+\frac{N}{N-1}\sinh^2\xi_1\right)\\
&& \left.+\cdots+\sinh^2\xi_M\cdot\prod_{i=1}^{M-1}\left(1+\frac{N}{N-1}\sinh^2\xi_i\right)\right],
\end{eqnarray*}
\begin{eqnarray*}
&& =-\left(\prod_{i=1}^M\cosh\xi_i\right)^2
+\sum_{S\subset\{1,\ldots,M\}}\left(\frac{N}{N-1}\right)^{\#\,S}\prod_{i\in S}\sinh^2\xi_i\\
&& -\frac1{N-1}\sum_{\emptyset\ne S\subset\{1,\ldots,M\}}\left(\frac{N}{N-1}\right)^{\#\,S-1}
\prod_{i\in S}\sinh^2\xi_i\\
&& =-\prod_{i=1}^M\left(1+\sinh^2\xi_i\right)+1
+\sum_{\emptyset\ne S\subset\{1,\ldots,M\}}\left(\frac{N}{N-1}\right)^{\#\,S-1}
\prod_{i\in S}\sinh^2\xi_i\\
&&=\sum_{S\subset\{1,\ldots,M\},\,\# S\ge2}\left[\left(\frac{N}{N-1}\right)^{\#\,S-1}-1\right]
\prod_{i\in S}\sinh^2\xi_i.
\end{eqnarray*}
Thus
\begin{eqnarray*}
&& \sigma^2\le\sum_{S\subset\{1,\ldots,M\}}\left(\left(\frac{N}{N-1}\right)^{M-1}-1\right)
\prod_{i\in S}\sinh^2\xi_i\\
&& =\left[\left(\frac{N}{N-1}\right)^{M-1}-1\right]
\prod_{i=1}^M\left(1+\sinh^2\xi_i\right)=
\left[\left(\frac{N}{N-1}\right)^{M-1}-1\right]\prod_{i=1}^M\cosh^2\xi_i.\\
&& \sigma\le\sqrt{\left[\left(\frac{N}{N-1}\right)^{M-1}-1\right]}\prod_{i=1}^M\cosh\xi_i.
\end{eqnarray*}

\begin{conclusion}\label{conc:HyVectors}
In the hyperbolic space $\mathbf{H}^{N-1}$, the hyperbolic cosine of the `hyperbolic arc' $\xi$ made by $M$ moves of given `hyperbolic arcs' $\xi_1,\xi_2,\ldots,\xi_M$, modeled by the above Markov chain, is with almost full probability, near%
\begin{equation}\label{eq:HyVectors}
\prod_{i=1}^{M}\cosh\xi_i.
\end{equation}
With relative deviation $O\left(\sqrt{M}/\sqrt{N}\right)$.)
\end{conclusion}
\bigskip
\textbf{What does conclusion (\ref{conc:HyVectors}) tell us?}
Firstly, as with the sphere (\S\ref{sc:SphereVectors}), (\ref{eq:HyVectors}) agrees with the `flat' case \S\ref{sc:Sum} when the $\xi$'s are small: here $\cosh\xi\sim 1+\frac12\xi^2$ and again multiplying
these corresponds approximately to adding the $\xi^2$'s.
\pa
But when the $\xi$ are large, $\cosh\xi\sim\frac12e^{\xi}$ and multiplying
these corresponds approximately to adding the $\xi$'s themselves, and subtracting $\log_2M$.
\pa
So, in the hyperbolic space, for large $\xi$ `the distances combine as if they were all about on the same line'.


\begin{thebibliography}{99}
\bibitem {markov2000}E. Behrends, \emph{Introduction to Markov Chains},
Springer, 2000.

\bibitem {bell16}J. Bell, Trace class operators and Hilbert-Schmidt operators,
Technical report, April 18, 2016, 26pp. Available on Sem0antic Scholar at https://www.semanticscholar.org/.

\bibitem {cl-2020}Y. Censor, E. Levy, An analysis of the superiorization method
via the principle of concentration of measure, \emph{Applied Mathematics and
Optimization}, accepted for publication, (2019).

It can be found as item [162] on: http://math.haifa.ac.il/yair/censor-recent-pubs.html\#bottom. The paper contains a brief introduction to the superiorization methodology and its history.

\bibitem {song97}D. Song and A. Gupta, $L_{p}$-norm uniform distribution,
\emph{Proceedings of the American Mathematical Society} \textbf{125},
(1997), 595--601.

\end{thebibliography}
\end{document}